\documentclass[letterpaper, 10 pt, conference]{ieeeconf}  

\IEEEoverridecommandlockouts                              

\usepackage{customizations}

\title{\LARGE \bf
Synthesis of the Transient Behaviour of Optimization Algorithms for Strongly Convex Objective Functions}

\author{Adel Abdelsamed$^{1}$, Dennis Gramlich$^{1}$, Christian Ebenbauer $^{1}$ and Carsten W. Scherer $^{2}$
\thanks{This work has been submitted to the IEEE
for possible publication. Copyright may be
transferred without notice, after which this
version may no longer be accessible.}
\thanks{$^{1}$Adel Abdelsamed, Dennis Gramlich and Christian Ebenbauer are with the Chair of Intelligent Control Systems, RWTH Aachen University, 52074 Aachen, Germany \{\tt\small adel.abdelsamed, dennis.gramlich, christian.ebenbauer\}@ic.rwth-aachen.de}%
\thanks{$^{2}$Carsten W. Scherer is with the Chair of Mathematical Systems Theory, University of Stuttgart, 70569 Stuttgart, Germany \tt\small carsten.scherer@imng.uni-stuttgart.de}%
\thanks{This work was supported by the German Research Foundation DFG under Grant EB425/9-1.}
}

\begin{document}

\maketitle
\thispagestyle{empty}
\pagestyle{empty}

\begin{abstract} 
We study the transient behavior of first-order optimization methods for unconstrained minimization of smooth and strongly convex objective functions. These methods are modelled as linear systems over finite-time intervals embedded in feedback with the gradient nonlinearity. By characterizing the function class through discrete-time Zames–Falb integral quadratic constraints, we employ a robust $\mathcal{H}_2$ control framework to obtain convex semidefinite programs certifying tight worst-case bounds on the terminal distance to the optimizer. 
Building on the same formulation, we derive a convex synthesis procedure for a transient-optimal method that achieves the optimal finite-horizon bounds and guarantees error reduction over any prescribed number of steps. The framework naturally extends to delayed gradient measurements, enabling both transient performance analysis and synthesis of optimal algorithms under feedback delays.
\end{abstract}
\section{INTRODUCTION}
\noindent Classical analyses of convex optimization algorithms emphasize asymptotic convergence, whereas many applications are dominated by finite-time, or transient, performance. Model predictive control is a prominent example: the optimization problem must be solved within a fixed sampling interval, meaning only a limited number of iterations are available before a decision must be made. In such settings, the transient behaviour of the algorithm — how quickly and reliably it approaches the optimizer in the first few iterations — is of greater practical relevance than its eventual asymptotic rate.

To address this gap in transient analysis, this paper aims to unify two complementary research streams. The first stream was initiated by the seminal work of Drori and Teboulle, who introduced the Performance Estimation Problem (PEP) framework for analyzing worst-case behaviour of first-order methods applied to smooth convex functions and provided approximate solutions~\cite{drori2012performancefirstordermethodssmooth}. Subsequent work by Taylor et al. delivered exact analytical PEP solutions and the Performance Estimation Toolbox (PESTO), enabling systematic semidefinite programming–based analysis over a broad range of performance criteria~\cite{taylor2016smoothstronglyconvexinterpolation, PESTO}. This line of research culminates in the analytically defined Information-Theoretic Exact Method (ITEM) for smooth strongly convex minimization, which achieves the oracle-complexity lower bound for black-box first-order methods on this problem class ~\cite{drori2021oraclecomplexitysmoothstrongly,taylor2022optimalgradientmethodsmooth}.

The second research stream, closer in spirit to this paper, adopts a control-theoretic framework for analyzing and designing optimization algorithms.  In this view, optimization methods are treated as dynamical systems and studied with tools from control theory. Earlier work has designed continuous-time optimization algorithms as smooth feedback systems, establishing convergence through passivity-based arguments~\cite{DURR2012291}, and has extended classical momentum methods, which are naturally described as second-order systems, to higher-order integrator dynamics with convergence guarantees provided by robust control techniques~\cite{multidimensionalHB}. Robust control techniques such as integral quadratic constraints (IQCs) have proven particularly powerful in this context, as also demonstrated by Lessard et al., who exploit Luré systems and IQCs to analyze the worst-case asymptotic convergence of first-order methods~\cite{Lessard_2016}. More recently, convex synthesis approaches based on dynamic IQCs have been developed to automatically design accelerated first-order methods with optimal guaranteed asymptotic rates, even in the presence of nontrivial dynamics such as delays ~\cite{convex_synthesis, scherer2025tutorial}. 

Within this control-theoretic stream, the transient behaviour of accelerated methods has also been studied in~\cite{mohammadi2023transient}. Mohammadi et al. attribute the transient growth of Nesterov's and the heavy-ball method on convex quadratic objectives to nonnormal dynamics and, using IQCs, establish analytical upper and lower bounds on the peak deviation from the optimizer for the general class of smooth and strongly convex objectives. These results quantify the transient cost of acceleration for given algorithms using a specific notion of transient growth, but does not address algorithm synthesis or additional dynamics such as delayed gradient measurements.  

The main contributions of this paper are two-fold. First, we propose a control-theoretic framework, based on discrete-time Zames--Falb IQCs, to analyze the worst-case \emph{transient} performance of first-order algorithms in the input--output domain, for general finite-horizon quadratic performance objectives. The resulting convex semidefinite programs (SDPs) provide tight finite-horizon upper bounds that coincide with those obtained from exact PEP formulations. Second, we use Zames--Falb multipliers to synthesize a transient-optimal linear time-varying first-order method. Notably, the synthesized parameters recover the information-theoretic exact method (ITEM) as a special case, demonstrating that optimal finite-horizon performance can also be realized via IQC-based synthesis. A key feature of our synthesis is the ability to explicitly model delays in the gradient measurements, which are unavoidable in practical applications due to communication latencies or computation times and can severely degrade the critical first few iterations of an algorithm.

The paper is organized as follows. Section~\ref{sec:algorithm_class} first introduces the class of smooth strongly convex objective functions, then establishes the finite-time horizon systems framework and thereby the considered class of first-order optimization algorithms from a control-theoretic perspective. To characterize the gradient nonlinearities, the notion of IQCs is also introduced. In Section~\ref{sec:problem_formulation}, the formal problem is introduced. 
Section~\ref{sec:robust_analysis_synthesis} then presents the analysis problem and a convexification of the synthesis problem, while Section~\ref{sec:simulation} numerically validates the obtained upper bounds for different well-established first-order algorithms. Lastly, in Section~\ref{sec:conclusion} a summary and an outlook are provided.  

\textbf{Notation:}
The sets $\mathbb{N}$, $\mathbb{N}_0 = \mathbb{N} \cup \{0\}$, and $\mathbb{R}$ denote the natural numbers, nonnegative integers, and real numbers, respectively. Vectors are denoted by lower-case letters, e.g., $v\in\mathbb{R}^n$, with Euclidean norm $\|v\|$. The all-ones vector is $\allones_\tau$, and $\unitvec_i$ is the $i$-th unit vector. Matrices are upper-case letters, while block matrices are \textbf{bold} upper-case. For a block matrix $\boldsymbol{A}$, the $(i,j)$-th block is $\boldsymbol{A}(i,j)$, while $A_{ij}$ denotes the $(i,j)$-th entry of an ordinary matrix. For block-diagonal $\boldsymbol{A}$, $\boldsymbol{A}(i)$ denotes the $i$-th diagonal block. The block-diagonal matrix with blocks $A(1),\ldots,A(k)$ is $\diag(A(1),\ldots,A(k))$, while $\col(A(1),\ldots,A(k))$ vertically stacks matrices with matching column dimensions.  
The identity and zero matrices of dimension $d$ are $I_d$ and $0_d$, respectively. The dimensions are omitted when clear.

For $A,B\in\mathbb{S}_n$ (the $n\times n$ symmetric matrices), $A \succ B$ ($A \succeq B$) means $A-B$ is positive (semi-)definite, analogously for $\prec$, $\preceq$; matrix inequalities $\leq$, $\geq$ are entrywise. $A\otimes B$, $\trace(A)$, and $A^\transpose$ denote Kronecker product, trace, and transpose, $\He(A) := A + A^\transpose$, and $(\bullet)^\transpose P x$ abbreviates $x^\transpose P x$. In symmetric matrices, $\star$ marks blocks determined by symmetry. The inner product is $\langle x,y\rangle = x^\transpose y$ for vectors and $\langle A,B\rangle = \trace(A^\transpose B)$ for matrices.

Lastly, $\mathbb{E}[x]$ and $\operatorname{Cov}(x)$ denote the
expectation and covariance matrix of a random vector $x$,
respectively.
\section{Class of objective functions and algorithms}\label{sec:algorithm_class}
\noindent We consider methods solving
\begin{mini*}
    {z\in \R^\optdim}{f(z),} {}{}
\end{mini*}
where the problem dimension $\optdim$ is known, but the exact function $f$ is assumed unknown. Instead, it is assumed that $f$ belongs to the class of smooth and strongly convex functions.
\begin{definition}
Let $z \in \R^\optdim$ and let $\strgcvxParam$, $\smoothParam$ be two real constants with $0 < \strgcvxParam < \smoothParam < \infty$. Any function belonging to the set 
\begin{equation*}
\begin{split}
    \smoothstrgcvxS_{\strgcvxParam, \smoothParam} := \big\{& f:\R^\optdim \rightarrow \R \,\big|\, 
     f(z) - \frac{\strgcvxParam}{2} \|z\|^2 \\
    & \text{and } \frac{\smoothParam}{2} \|z\|^2 - f(z) \text{ are convex} \big\}
\end{split}
\end{equation*}
is said to be $\strgcvxParam$-strongly convex and $\smoothParam$-smooth. The ratio $\kappa = \frac{\smoothParam}{\strgcvxParam}$ is called the condition number of the class.
\end{definition}
Both the strong convexity parameter $\strgcvxParam$ and the smoothness parameter $\smoothParam$ are assumed to be known. Note that the special case $\strgcvxParam = \smoothParam$ is excluded from the analysis, as it results in trivial quadratic functions. Functions $f \in \smoothstrgcvxS_{\strgcvxParam, \smoothParam}$ are differentiable and admit a unique global minimizer $z_\star = \argmin_z\, f(z) \in \R^\optdim$ such that $\nabla f(z_\star) = 0$~\cite{scherer2023robustexponentialstabilityinvariance}.

For analytical convenience, we consider the subclass of functions in $\smoothstrgcvxS_{\strgcvxParam,\smoothParam}$ whose unique minimizer is the origin:
$ \smoothstrgcvxS_{\strgcvxParam, \smoothParam}^0 = \left\{f\in \smoothstrgcvxS_{\strgcvxParam, \smoothParam}\,\middle|\, \nabla f(0) = 0\right\}.$
Subsequent sections demonstrate that results established for $\smoothstrgcvxS_{\strgcvxParam, \smoothParam}^0$ extend naturally to the broader class $\smoothstrgcvxS_{\strgcvxParam, \smoothParam}$.

\noindent In prior work~\cite{Lessard_2016,scherer2023robustexponentialstabilityinvariance,scherer2025tutorial}, a class of first-order optimization algorithms has been modelled as feedback interconnections between a linear system represented in state space and the gradient nonlinearity. In particular, the linear part is described by
\begin{equation}
\begin{pmatrix}
    x(t+1) \\ z(t) \\ y(t)
\end{pmatrix}
=
\begin{pmatrix}
    A\otimes\eye_\optdim & B\otimes\eye_\optdim\\
    C_z\otimes\eye_\optdim & 0 \\
    C_y\otimes\eye_\optdim & 0
\end{pmatrix}
\begin{pmatrix}
    x(t) \\ w(t)
\end{pmatrix},
\label{eq:state_space_opt_alg}
\end{equation}
with the feedback relation \(w(t)=\nabla f(z(t))\). Here, \(x(t)\in\R^{nd}\) denotes the state. The two output signals serve distinct purposes: \(z(t)\in\R^\optdim\) specifies points at which the gradient is evaluated, whereas \(y(t)\in\R^\optdim\) denotes the iterates generated by the optimization algorithm. The Kronecker products with $\eye_\optdim$ arise because the considered optimization algorithms apply an identical scalar update rule to each coordinate, excluding methods that scale coordinates differently, such as quasi-Newton methods.
\begin{example}\label{ex:nagm_example}
Consider the recursion defined by Nesterov's accelerated gradient method (NAGM),
\begin{equation*}
\begin{aligned}
    y(t+1) &= z(t) - \alpha \nabla f(z(t)), \\
    z(t) &= (1+\beta)y(t) - \beta y(t-1),
\end{aligned}
\end{equation*}
where \(\alpha \in \left(0, \frac{2(1+\beta)}{L}\right)\) is the step size, \(\beta\in(0,1)\) the momentum parameter, and we assume the method is initialized by $y(-1) = y(0)$. This initialization convention implies $z(0) = (1+\beta)y(0) - \beta y(-1) = y(0)$, which corresponds to applying a gradient descent step for the first iteration. A state-space description of this method is obtained with
\begin{equation}
\begin{split}
    A &= \begin{pmatrix}
        1+\beta & -\beta \\
        1 & 0
    \end{pmatrix}, \qquad
    B = \begin{pmatrix}
        -\alpha \\
        0
    \end{pmatrix}, \\
    C_z &= \begin{pmatrix}
        1+\beta & -\beta
    \end{pmatrix}, \qquad
    C_y = \begin{pmatrix}
        1 & 0
    \end{pmatrix}.
\end{split}
\label{eq:nagm_ss_matrices}
\end{equation}
Notably, $C_z \neq C_y$, reflecting the fact that the gradient 
evaluation point $z(t)$ and the algorithm's iterate $y(t)$ are not the same signal for NAGM.
\end{example}

Since our focus is on the transient behaviour of optimization algorithms, we adopt a different perspective on this feedback interconnection. Specifically, we consider algorithms over a finite-time horizon and describe their linear dynamics through an \emph{input-output} representation. To this end, we first introduce the notion of a signal defined on a finite horizon.
\begin{definition}\label{def:signal}
Let $\horizonSet_\tau := \{0, 1, \dots, \tau\} \subset \N_0$  be a discrete time interval, where $\tau$ is the signal horizon length. A signal $\zSignal$ is a mapping $\zSignal :\horizonSet_\tau \rightarrow \R^\optdim$. The signal space $\signalSpace_{2e}^\optdim(\horizonSet_\tau)$ is the set of all signals $\zSignal$ on the horizon $\horizonSet_\tau$.
\end{definition}

Of particular interest is the pointwise action of the gradient map
$\nabla f$ on signals defined over the finite horizon. For a fixed
$f\in\smoothstrgcvxS_{\strgcvxParam, \smoothParam}$, we denote the lifted gradient operator
\(\Delta_f :\ell_{2e}^{d}(\horizonSet_\tau)\rightarrow\ell_{2e}^{d}(\horizonSet_\tau)\)
by
\(
\Delta_f(\boldsymbol{z})
:=
\operatorname{col}\bigl(
\nabla f(z(0)),\ldots,\nabla f(z(\tau))
\bigr).
\)
Thus, $\nabla f:\mathbb{R}^d\to\mathbb{R}^d$ denotes the gradient
map acting on individual points $z(t)$, whereas $\Delta_f$ denotes its
pointwise extension to finite-horizon signals $\zSignal$.
 In control-theoretic terms, $\Delta_f$ defines a repeated nonlinearity.

With this finite-time horizon perspective, optimization algorithms are modelled as linear dynamical systems between signal spaces.
\begin{definition}
A linear system on a discrete interval $\horizonSet_\tau$ is a linear mapping $\G: \signalSpace_{2e}^{\inputdim}(\horizonSet_\tau) \to \signalSpace_{2e}^{\outputdim}(\horizonSet_\tau)$, where $\inputdim$ and $\outputdim$ are the dimensions of the input and output signals, respectively. The set of linear systems from $\signalSpace_{2e}^{\inputdim}(\horizonSet_\tau)$ to $\signalSpace_{2e}^{\outputdim}(\horizonSet_\tau)$ is denoted by $\linsysSet(\signalSpace_{2e}^{\inputdim}(\horizonSet_\tau), \signalSpace^{\outputdim}_{2e}(\horizonSet_\tau))$.
\label{def:lin_io_systems}
\end{definition}
Since optimization algorithms exclusively utilize past and current gradients to compute the next iterate, structural constraints on the mapping $\G$ must be imposed.  
\begin{lemma}(cf.~\cite{adel2025})
A linear system $\G: \signalSpace_{2e}^{\inputdim}(\horizonSet_\tau) \to \signalSpace_{2e}^{\outputdim}(\horizonSet_\tau)$ is causal if and only if its matrix representation has a block-lower-triangular structure
\begin{equation*}
    \G = 
\begin{pmatrix}
G(0,0) & \zeros& \cdots & \zeros\\
G(1, 0) &G(1,1)  & \cdots & \zeros\\
\vdots & \vdots & \ddots & \vdots \\
G(\tau,0) & G(\tau,1)  & \cdots & G(\tau,\tau) 
\end{pmatrix}
\end{equation*}
where each block $G(t,j) \in \R^{\outputdim \times \inputdim}$ represents the contribution from input at time $j$ to the output at time $t$.
\label{theorem:causality}
\end{lemma}
\begin{proof}
The necessity and sufficiency follow directly from the definition of causality.
\end{proof}
A refinement of the causality property is the notion of relative degree, which captures how many time steps elapse before an input can influence the output.
\begin{remark}
A linear system $\G$ has \emph{relative degree} $r \in \N_0$ if its matrix is $r$-strictly lower triangular, i.e., $G(t,j)=0$ for all $t \le j+r-1$. Equivalently, an input at time $j$ affects the output only from time $j+r$ onward.
\end{remark}

For the considered class of optimization algorithms, the input and output dimensions coincide with the dimension of the decision variable, i.e., $\inputdim = \outputdim := \optdim$, and the block matrices take the form $G(t,j) = g_{t,j}\eye_\optdim$ for scalars $g_{t,j} \in \R$. The set of such linear systems with relative degree $r \in \N_0$ and horizon $\tau$ is denoted by $\linsysSet^r(\optdim, \tau)$, where $r = 0$ recovers the causal case of Lemma~\ref{theorem:causality}.

To derive the input-output representation of the general state-space representation~\eqref{eq:state_space_opt_alg} the recursion is unrolled over the horizon $\tau$, yielding
\begin{equation}
\begin{split}
    \textbf{z} = \mathcal{O}_{C_z} y(0) + \textbf{G}_{\zSignal\wSignal}\textbf{w}, \qquad 
    \textbf{y} = \mathcal{O}_{C_y} y(0) + \textbf{G}_{\ySignal\wSignal}\textbf{w},
\end{split}
\label{eq:general_io_from_ss}
\end{equation}
where $\mathcal{O}_{C_z}$ and $\mathcal{O}_{C_y}$ are the 
observability matrices for output $\zSignal$ and $\ySignal$, and the blocks of $\G_{\zSignal\wSignal}$ and $\G_{\ySignal\wSignal}$ are given by
$G_{\zSignal\wSignal}(t,j) = C_z A^{t-j-1} B$ and $G_{\ySignal\wSignal}(t,j) = C_y A^{t-j-1} B$ for $t > j$, and zero otherwise. Substituting the state-space matrices of NAGM~\eqref{eq:nagm_ss_matrices} in~\eqref{eq:general_io_from_ss} gives
\begin{equation*}
    \zSignal = \allones_\tau \otimes y(0) + \G^{\text{NAGM}}_{\zSignal\wSignal}\,\wSignal,
    \qquad
    \ySignal = \allones_\tau \otimes y(0) + \G^{\text{NAGM}}_{\ySignal\wSignal}\,\wSignal,
\end{equation*}
with the operators 
\begin{equation*}
\begin{aligned}
    G^{\text{NAGM}}_{\zSignal\wSignal}(t,j) 
    & = -\alpha\,\frac{1-\beta^{t-j+1}}{1-\beta}\eye_\optdim 
    & \qquad t>j,\\
    G^{\text{NAGM}}_{\ySignal\wSignal}(t,j) 
    & = -\alpha\,\frac{1-\beta^{t-j}}{1-\beta}\eye_\optdim 
    & \qquad t>j.
\end{aligned}
\end{equation*}
The term $\allones_\tau\otimes y(0)$ arises from the integral action in the loop, which constitutes a minimal requirement for convergence, and is not specific to NAGM~\cite{scherer2025tutorial}. 
It carries an important implication: if the iterates are initialized at the optimum $y(0) = z_\star$, they remain there, consistent with 
the first-order optimality condition $\nabla f(z_\star) = 0$.

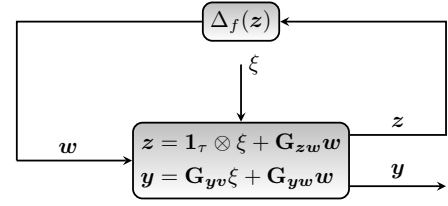
\begin{figure}[t]
    \centering
    \resizebox{0.35\textwidth}{!}{\tikzstyle{block} = [draw,rectangle,thick,minimum height=1em,minimum width=1em,
    shade, shading=axis, top color=lightgray, bottom color=white, rounded corners=.15cm]
    
\tikzstyle{sum} = [draw,circle,inner sep=0mm,minimum size=0mm]
    
\tikzstyle{connector} = [-,thick]
    
\tikzstyle{arrowconnector} = [-stealth,thick]

\begin{tikzpicture}[scale=1.0, auto]
    
    \node[block] (system) {
        $\begin{aligned}
            \zSignal &= \allones_{\tau}\otimes \xi + \G_{\zSignal\wSignal}\wSignal \\
            \ySignal &= \G_{\perfOut\perfIn} \xi + \G_{\ySignal\wSignal}\wSignal
        \end{aligned}$
    };
    
    \node[block, above=1.25cm of system] (gradient) {
        $\liftf(\zSignal)$
    };
    
    \coordinate (input_point) at (-3.5, 0);
    \draw[connector] (input_point) -- ([xshift=-0.1cm]system.west);
    \draw[arrowconnector] ([xshift=-0.1cm]system.west) -- (system.west);
    \node[above] at ([xshift=-1.0 cm]system.west) {$\wSignal$};
    
    \coordinate (z0_point) at (0, 1.5);
    \draw[arrowconnector] (z0_point) -- (system.north);
    \node[right] at (0, 1.5) {$\xi$};
    
    \draw[arrowconnector] ([yshift=0.4cm]system.east) -- ++(1.5,0) |- (gradient.east);
    \node[above] at ([xshift=0.75cm, yshift=0.4cm]system.east) {$\zSignal$};
    
    \draw[arrowconnector] ([yshift=-0.4cm]system.east) -- ++(1.5,0);
    \node[above] at ([xshift=0.75cm, yshift=-0.4cm]system.east) {$\ySignal$};
    
    \draw[connector] (gradient.west) -| (input_point);
    
\end{tikzpicture}}
    \caption{Feedback interconnection between gradient nonlinearity and linear plant.}
    \label{fig:feedback_interconnection}
\end{figure}

The derived input-output representation of NAGM motivates the 
following general algorithm class. Let $\horizonSet_{\tau-1}$ be a finite 
horizon and let $\xi := y(0)\in\R^\optdim$ be the initial iterate. The class of 
first-order optimization algorithms considered in this paper admits the 
input-output description
\begin{equation}
\begin{aligned}
    \zSignal &= \allones_\tau\otimes \xi + \G_{\zSignal\wSignal}\wSignal, \\
    \ySignal &= \G_{\perfOut\perfIn} \xi + \G_{\ySignal\wSignal}\wSignal, \quad     \wSignal &= \liftf(\zSignal),
\end{aligned}
\label{eq:io_feedback_interconnection}
\end{equation}
as depicted in Figure~\ref{fig:feedback_interconnection}, where 
$\G_{\zSignal\wSignal}$ and $\G_{\ySignal\wSignal}$ belong to 
$\linsysSet^{r}(\optdim,\tau-1)$. The signal $\zSignal$ specifies the 
gradient evaluation points, while $\ySignal$ is the performance output.
As seen in Example~\ref{ex:nagm_example}, these signals are generally distinct, and thus maintaining this separation allows momentum-based methods to be represented within the proposed algorithm class.
The operator $\G_{\perfOut\perfIn}$ is left general to allow $\ySignal$ 
to represent arbitrary performance criteria, such as the final iterate or 
a final gradient. Importantly, the class of algorithms represented by~\eqref{eq:io_feedback_interconnection} extends the class captured by the state-space description of~\eqref{eq:state_space_opt_alg}, since it additionally includes time-varying algorithms, such as gradient descent with adaptive step sizes. Moreover, allowing $r > 1$ in $\G_{\zSignal\wSignal}, \G_{\ySignal\wSignal} \in \linsysSet^r(\optdim, \tau-1)$ captures algorithms with delayed gradient feedback.

Since the objective function $f$ is unknown, the feedback interconnection in Figure~\ref{fig:feedback_interconnection} cannot be analyzed directly. Before formally stating the problem in 
Section~\ref{sec:problem_formulation}, we first introduce the tools needed to handle this uncertainty. 
By restricting $f$ to $\smoothstrgcvxS_{\strgcvxParam,\smoothParam}$, the gradient belongs to the set of $(\strgcvxParam,\smoothParam)$-slope-restricted gradients. Using IQCs, a systematic treatment of such feedback systems is possible by replacing the uncertain, possibly nonlinear gradient $\nabla f$ with quadratic constraints relating its inputs and outputs.

\begin{definition}
Let $\iqcmultiplier$ be a symmetric matrix partitioned as
\[
\iqcmultiplier = \begin{pmatrix}
\iqcmultiplier_{11} & \iqcmultiplier_{12} \\
\iqcmultiplier_{12}^\transpose & \iqcmultiplier_{22}
\end{pmatrix} \in \R^{2d(\tau+1)\times2d(\tau+1)}.
\]
An (nonlinear) operator $\uncertainty: \signalSpace_{2e}^\optdim(\horizonSet_\tau) \to \signalSpace_{2e}^\optdim(\horizonSet_\tau)$ is said to satisfy the IQC defined by $\iqcmultiplier$ if
\begin{equation}
0 \leq \begin{pmatrix}
\uSignal \\ \Delta(\uSignal)
\end{pmatrix}^\transpose \iqcmultiplier \begin{pmatrix}
\uSignal \\ \Delta(\uSignal)
\end{pmatrix}
\label{eq:iqc_inequality}
\end{equation}
holds for all $\uSignal \in \signalSpace_{2e}^\optdim(\horizonSet_\tau)$. $\iqcmultiplier$ is called the IQC multiplier.
\end{definition}

To establish IQCs for the considered class of $(\strgcvxParam,\smoothParam)$-slope restricted gradients, a special class of matrices called \emph{doubly hyperdominant} matrices is defined.
\begin{definition}
A square matrix $\multiplierM$ of dimension $\tau$ is called doubly hyperdominant if $\multiplierM \in \dhmatricesSet_{\tau}$, where
\begin{equation*}
\dhmatricesSet_{\tau} :=
\left\{
\multiplierM \in \R^{\tau \times \tau} \;\middle|\;
\begin{aligned}
& \multiplierM \allones_{\tau} \geq 0,\quad \allones_{\tau}^\transpose \multiplierM \geq 0, \\
& \multiplierM_{ij} \leq 0,\quad \forall i\neq j
\end{aligned}
\right\}.
\end{equation*}

\end{definition}

The next lemma introduces discrete-time Zames-Falb IQCs characterizing $(\strgcvxParam,\smoothParam)$-slope restricted gradients through doubly hyperdominant matrices.

\begin{lemma}(cf.~\cite{FETZER2017}). Let $f : \R^\optdim \rightarrow \R$ with 
$f \in \smoothstrgcvxS_{\strgcvxParam,\smoothParam}^0$. Then
\begin{equation}
\label{eq:quadratic_form_ozf_multiplier}
0 \leq 
\begin{pmatrix} \zSignal \\  \liftf(\zSignal) \end{pmatrix}^\transpose \,\boldsymbol{\Pi}_{\multiplierM}\, \begin{pmatrix} \zSignal \\  \liftf(\zSignal) \end{pmatrix}
\end{equation}
for all $\zSignal \in \signalSpace_{2e}^\optdim(\horizonSet_\tau)$, where
\begin{equation*}
\boldsymbol{\Pi}_{\multiplierM} := 
\begin{pmatrix}
\smoothParam\eye & -\eye \\
-\strgcvxParam\eye & \eye 
\end{pmatrix}^\transpose
\begin{pmatrix}
\zeros & \boldsymbol{\multiplierM} \\
\boldsymbol{\multiplierM}^\transpose & \zeros
\end{pmatrix}
\begin{pmatrix}
\smoothParam\eye & -\eye \\
-\strgcvxParam\eye & \eye 
\end{pmatrix},
\end{equation*}
with $\boldsymbol{\multiplierM} = \multiplierM \otimes \eye_{\optdim}$ for $\multiplierM \in \dhmatricesSet_{\tau+1}$. Here, $\eye$ and $\zeros$ denote identity and zero matrices of dimension $(\tau+1)\optdim$.
\end{lemma}

\section{Problem Formulation}
\label{sec:problem_formulation}
\noindent Classical analyses evaluate the performance of optimization algorithms through the worst-case convergence rate $\rho_{wc}(\kappa) \in (0,1)$ with the associated bound
\begin{equation}
    \|y(t) - z_\star\|^2 \leq c(\kappa)^2\,\rho_{wc}^{2t}(\kappa)\,\|\xi - z_\star\|^2,
\end{equation}
which must be satisfied for every $f \in \smoothstrgcvxS_{\strgcvxParam,\smoothParam}$ and every initial condition $\xi$ at each stage $t$. 
This asymptotic bound separates two factors with fundamentally different roles. While $\rho_{wc}^{2t}(\kappa)$ dominates
for large $t$, finite-horizon performance is strongly influenced by the generally unknown constant $c(\kappa)$, which can be large. In this paper, we propose a new bound on the terminal distance, which merges the asymptotic and transient effects.
The considered finite-horizon bound is of the form
\begin{equation}\label{eq:transient_bound}
    \|y(\tau) - z_\star\|^2 \leq \gamma(\kappa,\tau)^2\,\|\xi - z_\star\|^2,
\end{equation}
where $\gamma(\kappa,\tau)^2$ is the worst-case squared amplification factor over a horizon of $\tau$
iterations. Unlike $\rho_{wc}^{2t}(\kappa)$, a convergent algorithm might have  $\gamma(\kappa,\tau)^2\geq 1$ for small $\tau$. In general, transient performance of an algorithm depends on the algorithm parameters, the objective function, and the initial condition. For optimization algorithms representable by the 
class~\eqref{eq:io_feedback_interconnection}, modelling their linear 
dynamics over finite-time intervals renders both the \emph{analysis} 
of $\gamma(\kappa,\tau)^2$ for a given algorithm and the 
\emph{synthesis} problem tractable via convex SDPs.

In particular, in the synthesis problem, we seek algorithm parameters that minimize a worst-case performance criterion over all $f\in\smoothstrgcvxS_{\strgcvxParam,\smoothParam}$, while measuring the sensitivity to the initial condition through an average quadratic criterion. To this end, the initial iterate $\xi$ is modelled as an exogenous random vector with finite first and second moments satisfying
\(
\mathbb{E}[\xi]=z_\star, \operatorname{Cov}(\xi)=\boldsymbol{\Sigma}_0=I.\)
A quadratic performance output is used to capture objectives such as the terminal distance to the optimizer. In control theory, the expected value of such a quadratic output, when the system is driven by a random exogenous input with prescribed second-order moments, defines an $\etchH_2$-type performance criterion. The resulting synthesis problem is therefore a robust $\etchH_2$ problem, in which the unknown gradient map $\nabla f$ plays the role of the uncertainty. We formally define the problem.
\begin{problem}
Let $f \in \smoothstrgcvxS_{\strgcvxParam,\smoothParam}$ be an objective function with
minimizer $z_\star$, and let $\zSignal_\star := \col(z_\star,\ldots,z_\star)$. Furthermore, let
$\xi$ be a random vector satisfying
\(\mathbb{E}[\xi]=z_\star,\operatorname{Cov}(\xi)=\boldsymbol{\Sigma_0}=I\). The transient performance problem is of the form
\begin{mini}
    {\boldclosedloopG_{\zSignal\unceOut}, \boldclosedloopG_{\ySignal\unceOut}}
    {\sup_{f \in \smoothstrgcvxS_{\strgcvxParam, \smoothParam}}
     \mathbb{E}\langle \perfOut - \zSignal_\star,\,
     \boldsymbol{R}_{\perfOut\perfOut} (\perfOut - \zSignal_\star) \rangle}
    {\label{prob:min_sup_general_analysis}\tag{$\mathcal{H}_{2}$-opt}}{}
    \addConstraint{\unceIn}{=  \allones_{\tau} \otimes \xi
                              + \boldclosedloopG_{\zSignal\unceOut} \unceOut,}
    \addConstraint{\perfOut}{= \boldclosedloopG_{\perfOut\perfIn} \xi
                              + \boldclosedloopG_{\ySignal\unceOut} \unceOut,}
    \addConstraint{\unceOut}{=\liftf(\unceIn).}
\end{mini}
\end{problem}
where $\boldsymbol{R}_{\perfOut\perfOut}$ defines the performance output weighting matrix. For the terminal distance objective~\eqref{eq:transient_bound},
$\boldsymbol{R}_{\perfOut\perfOut} = \diag(\zeros, \dots, \zeros, \eye)$, such that
$\mathbb{E}\langle \perfOut - \zSignal_\star, \boldsymbol{R}_{\perfOut\perfOut}
(\perfOut - \zSignal_\star)\rangle = \mathbb{E}\|y(\tau) - z_\star\|^2$.

For $\optdim=1$, the criterion also admits a deterministic interpretation. Since the initial error has unit second moment, the resulting bound coincides with an upper bound on the squared transient amplification factor $\gamma(\kappa,\tau)^2$ in~\eqref{eq:transient_bound} over all initial errors satisfying \(||\xi-z_\star||^2\leq 1.\)
Subsequently, Corollary~\ref{lemma:dimred} shows that considering $\optdim=1$ entails no loss of generality.

The system dynamics are defined through the matrices
\begin{equation}
\begin{split}
\boldclosedloopG_{\unceIn\unceOut} &= \closedloopG_1 \otimes \eye_\optdim \in \linsysSet^1(\optdim, \tau-1),\\ \boldclosedloopG_{\perfOut\perfIn} &= \allones_\tau\otimes\eye_\optdim,\\ \boldclosedloopG_{\perfOut\unceOut} &= \unitvec_\tau \otimes \closedloopG_2\otimes I_\optdim, \quad \closedloopG_2 \in \R^{1 \times \tau}.
\label{eq:closed_loop_matrices}
\end{split}
\end{equation}
Furthermore, $\closedloopG_2$ is used to describe the terminal iterate $y(\tau)$ as a linear combination of the processed gradient history. Since both $\closedloopG_1$ and $\closedloopG_2$ use gradient information to generate iterates, they represent the algorithm (controller) parameters in this framework. Their roles differ, however: $\closedloopG_1$ aggregates gradients to shape update directions (similar to momentum), whereas $\closedloopG_2$ uses the full history to form an optimal estimate of the minimizer.

In the following, we restrict to $r=1$ and the terminal distance objective; both the proofs and the framework extend directly to $r>1$ and alternative objectives such as the terminal gradient.

Working with the class $\smoothstrgcvxS_{\strgcvxParam,\smoothParam}$ requires explicit knowledge of the optimizer $z_\star$. Therefore, the following lemma reformulates the problem in an equivalent way by shifting the minimizer to the origin.
\begin{lemma}
\label{lemma:h2_equivalence_error_coordinates}
Consider the Problem~\eqref{prob:min_sup_general_analysis}, and let $f \in \smoothstrgcvxS_{\strgcvxParam, \smoothParam}$ attain its minimum at $z_\star$. Define the error coordinates of the uncertainty channel $\delta \unceIn := \unceIn - \zSignal_\star$, the error coordinates of the performance channel as $\delta \perfOut:= \perfOut - \boldclosedloopG_{\perfOut\perfIn}z_\star$, and of the initial condition $\delta \xi := \xi - z_\star$. Then, $\mathbb{E}[\delta \xi] = 0, \operatorname{Cov}(\delta \xi) = \boldsymbol{\Sigma_0}$, and \eqref{prob:min_sup_general_analysis} is equivalent to
\begin{mini}
    {\boldclosedloopG_{\zSignal\unceOut}, \boldclosedloopG_{\ySignal\unceOut}}{\sup_{ h \in \smoothstrgcvxS_{\strgcvxParam, \smoothParam}^0} \mathbb{E}\langle \delta\perfOut, \boldsymbol{R}_{\perfOut\perfOut} \delta\perfOut \rangle}
    {\label{prob:min_sup_general_analysis_zero}\tag{$\mathcal{H}_{2}$-opt-0}}{}
    \addConstraint{\delta \unceIn}{= \allones_\tau \delta \xi +\boldclosedloopG_{\zSignal\unceOut} \unceOut,}
    \addConstraint{\delta\perfOut}{= \boldclosedloopG_{\perfOut\perfIn} \delta \xi + \boldclosedloopG_{\ySignal\unceOut}\unceOut,}
    \addConstraint{\unceOut}{= \lifth(\delta \unceIn).}
    \end{mini}
where $h(\bullet) := f(\bullet + z_\star)$. 

\noindent\emph{See Appendix~\ref{appendix:proof_lemma} for a proof.}
\end{lemma}
By reducing the considered set of functions to $\smoothstrgcvxS_{\strgcvxParam, \smoothParam}^0$, the use of the IQC multipliers presented in the previous subsection is enabled. In the following sections, notation is abused by assuming $ f \in \smoothstrgcvxS_{\strgcvxParam, \smoothParam}^0$ and dropping the $\delta$ in the deviation variables.

\section{Robust Analysis and Synthesis of the Terminal Distance Problem}
\label{sec:robust_analysis_synthesis}
\noindent Solving Problem~\eqref{prob:min_sup_general_analysis_zero} is challenging because it involves optimization over the infinite-dimensional function class $f \in  \smoothstrgcvxS_{\strgcvxParam,\smoothParam}^0$. By replacing the constraint \(\wSignal =  \liftf(\zSignal)\) 
with the IQC~\eqref{eq:iqc_inequality}
for a multiplier \(\boldsymbol{\multiplierM} = \multiplierM \otimes \eye_\optdim\) and \(\multiplierM \in \dhmatricesSet_\tau\), Problem~\eqref{prob:min_sup_general_analysis_zero} can be relaxed.
\begin{theorem}[cf. \cite{adel2025}]
\label{cor:io_sdp_relaxation}
Consider Problem~\eqref{prob:min_sup_general_analysis_zero} together with the
system matrices~\eqref{eq:closed_loop_matrices} and the performance weighting for the terminal distance.
Then, the problem
\begin{mini!}
{   \substack{
    \closedloopG_1 \in \linsysSet^1(1,\tau-1), \\
   \multiplierM \in \dhmatricesSet_\tau, \closedloopG_2,\multiplierP}
}
    {\trace(\multiplierP)}
    {}{}
     \addConstraint{\begin{aligned}
        &(\bullet)^\transpose  \begin{pmatrix}
         \eye_\optdim & \closedloopG_2\otimes I_\optdim \end{pmatrix} -(\bullet)^\transpose\multiplierP\begin{pmatrix}
            \eye & \zeros
        \end{pmatrix} +\\
    &\quad (\bullet)^\transpose 
        \iqcmultiplier_\multiplierM
        \begin{pmatrix}
        \allones_{\tau} \otimes \eye_\optdim & \closedloopG_1\otimes I_\optdim \\
        \zeros & \eye
        \end{pmatrix} \prec \zeros,
    \end{aligned}\label{eq:input_output_analysis_problem}}{}
     \addConstraint{\boldsymbol{\multiplierM}\label{eq:eq_2}}{= \multiplierM \otimes \eye_\optdim.}
     \addConstraint{}{\tag{$\mathcal{H}_{2}$-rel}\label{prob:UB_min_sup_io_problem2}}
\end{mini!}
provides an upper bound for the original problem, i.e.,
\(
    \text{\eqref{prob:min_sup_general_analysis_zero}}
    \;\le\;
    \text{\eqref{prob:UB_min_sup_io_problem2}}.
\)

\noindent\emph{See Appendix~\ref{appendix:proof_cor} for a proof.}
\end{theorem}

The key constraint of this problem is the matrix inequality~\eqref{eq:input_output_analysis_problem}, whose dimension is $\optdim \tau$. For high-dimensional problems (large $\optdim$), this quickly becomes computationally demanding. The next corollary provides an equivalent formulation with a dimensionality-reduced constraint that replaces the original matrix inequality with one that assumes $\optdim = 1$.
\begin{corollary}
    \label{lemma:dimred}
    Consider the matrix inequality~\eqref{eq:input_output_analysis_problem}. Since all matrices appearing in the inequality are Kronecker products with $I_\optdim$ and the inequality is linear in $\multiplierP$, any feasible $\multiplierP$ can be replaced by $\multiplierP = p\eye_\optdim$ without loss of generality. The matrix inequality~\eqref{eq:input_output_analysis_problem} can then be reduced to 
\begin{equation}    \label{eq:simplified_input_output_analysis_problem_sdp}  
    \begin{aligned}
        &(\bullet)^\transpose 
        \begin{pmatrix}
            -\strgcvxParam\smoothParam(\multiplierM+\multiplierM^\transpose) & \strgcvxParam\multiplierM + \smoothParam\multiplierM^\transpose \\
            \strgcvxParam\multiplierM^\transpose + \smoothParam\multiplierM & -(\multiplierM + \multiplierM^\transpose)
        \end{pmatrix}
        \begin{pmatrix}
        \allones_{\tau} & \closedloopG_1 \\
        \zeros & \eye
        \end{pmatrix} +\\
    &\quad (\bullet)^\transpose \begin{pmatrix}
        1 & \closedloopG_2 \end{pmatrix} - (\bullet)^\transpose p\begin{pmatrix}
            1 & \zeros
        \end{pmatrix} \prec \zeros.
    \end{aligned}
\end{equation}
    Furthermore, the matrix inequality~\eqref{eq:input_output_analysis_problem} is feasible if and only if the matrix inequality~\eqref{eq:simplified_input_output_analysis_problem_sdp} is feasible.
\end{corollary}

\begin{proof}Let $\Lambda$ denote the block-diagonal matrix whose diagonal blocks are the central matrices in~\eqref{eq:simplified_input_output_analysis_problem_sdp}, and let $\chi$ be the column stacked matrix of all right-hand matrices in~\eqref{eq:simplified_input_output_analysis_problem_sdp}. It is noted that~\eqref{eq:simplified_input_output_analysis_problem_sdp} can be equivalently written as $\chi^\transpose \Lambda \chi \prec 0$.

First, assume~\eqref{eq:simplified_input_output_analysis_problem_sdp} is feasible with solution $p$ and $\multiplierM$. Set $\multiplierP = p \eye_\optdim$ and substitute into~\eqref{eq:input_output_analysis_problem}. Then, \eqref{eq:input_output_analysis_problem} simplifies to
\(
(\chi \otimes \eye_\optdim)^\transpose (\Lambda \otimes \eye_\optdim) (\chi \otimes \eye_\optdim) 
= (\chi^\transpose \Lambda \chi) \otimes \eye_\optdim \prec 0,
\)
since $\chi^\transpose \Lambda \chi \prec 0$ and the the Kronecker product with $I_d$ preserves negative definiteness.

 For the other direction assume that the matrix inequality~\eqref{eq:input_output_analysis_problem} is feasible with $\multiplierP$. Then, perform right and left multiplication with $\diag(I_{\tau},I_{\tau},I_{\tau},1)\otimes e_1$ and its transpose, respectively. Note that $\diag(I_{\tau},I_{\tau},I_{\tau},1)\otimes e_1$ has full column rank. Multiplying a negative definite matrix from the right by a matrix of full column rank and from the left with its transpose yields a symmetric negative definite matrix. Further, define $p$ as 
$(I_{\tau} \otimes e_1)^\transpose \multiplierP (I_{\tau} \otimes e_1) =: p$, which yields exactly the matrix inequality~\eqref{eq:simplified_input_output_analysis_problem_sdp}.
\end{proof}

\subsection{Robust Analysis}
\noindent In the analysis setting, we assume a given optimization algorithm whose performance is to be evaluated over a horizon $\tau$. Consequently, the matrices $\closedloopG_1$ and $\closedloopG_2$ are known and fixed, and the matrix inequality in Problem~\eqref{prob:UB_min_sup_io_problem2} reduces to a linear matrix inequality (LMI) in the decision variables $\multiplierM$ and $p$. Thus, the variable $p$ is an upper bound on $\gamma^2(\kappa,\tau)$ from \eqref{eq:transient_bound}.
\subsection{Robust Synthesis}

\noindent If the algorithm matrices are decision variables, Problem~\eqref{prob:UB_min_sup_io_problem2} is nonconvex, because the multiplier and algorithm matrices enter bilinearly. Invoking Corollary~\ref{lemma:dimred}, Theorem~\ref{theorem:convexification_theorem1} convexifies the problem by introducing a bijective reparametrization that merges multiplier and algorithm matrices. To ensure that the inverse transformation preserves the admissible strictly causal algorithm structure, the multiplier is restricted to be causal. 
\begin{theorem}
\label{theorem:convexification_theorem1}
Restrict the multiplier in Problem~\eqref{prob:UB_min_sup_io_problem2} to
$\multiplierM \in \dhmatricesSet_\tau \cap \linsysSet^0(1,\tau-1)$. Then, the resulting problem is equivalent to the convex SDP
\vspace{-0.5em}
\begin{mini!}[1]
{\substack{\multiplierM \in \dhmatricesSet_{\tau}\cap \linsysSet^0(1,\tau-1),p,\\
\youlaQ_\multiplierM\in\linsysSet^1(1,\tau-1),\ \youlaQ_2}}
{p}{}{}
\addConstraint{\begin{pmatrix}
-2\strgcvxParam\smoothParam\allones_{\tau}^\transpose\multiplierM\allones_{\tau}-p & \star & \star\\
\smoothParam\multiplierM\allones_{\tau}-\strgcvxParam(\youlaQ_\multiplierM{-}\multiplierM)^\transpose\allones_{\tau} & \He(\youlaQ_\multiplierM{-}\multiplierM) & \star\\
1 & \youlaQ_2 & -1
\end{pmatrix} \prec \zeros. \label{eq:cvx_sdp_constraint}}{}
\end{mini!}
\begin{equation}
    \label{prob:UB_min_sup_io_problem3}\tag{$\mathcal{H}_{2}$-cvx}
\end{equation}
Feasibility of~\eqref{eq:cvx_sdp_constraint} guarantees that $\multiplierM$ is nonsingular, as shown in the proof. The controller parameters can then be obtained from any candidate solution $(\youlaQ_{\multiplierM},\youlaQ_2,\multiplierM)$ through the equations
\begin{equation}
\begin{split}    
        \closedloopG_1 &= \bigl((\smoothParam-\strgcvxParam)\eye +\strgcvxParam\multiplierM^{-1}\youlaQ_{\multiplierM}\bigr)^{-1}\multiplierM^{-1}\youlaQ_{\multiplierM}, \\
        \closedloopG_2 &= \youlaQ_2(\eye -\strgcvxParam\closedloopG_1).
\end{split}
\end{equation}
Consequently, the optimal value of
Problem~\eqref{prob:UB_min_sup_io_problem2} is no larger than
that of~\eqref{prob:UB_min_sup_io_problem3}, whereas the latter
is equal to the optimal value of
Problem~\eqref{prob:UB_min_sup_io_problem2} restricted to causal
multipliers.
\end{theorem}
 \begin{proof}
    First, the matrix inequality~\eqref{eq:simplified_input_output_analysis_problem_sdp} is rewritten as 
    \begin{equation*}
    \begin{split}
    (\bullet)^\transpose 
    &\begin{pmatrix}
        -p\unitvec_1\unitvec_1^\transpose & \zeros \\
        \zeros & 1
    \end{pmatrix} \begin{pmatrix}
        \allones_{\tau} & \closedloopG_1 \\
        1 & \closedloopG_2
    \end{pmatrix}+ \\ &(\bullet)^\transpose \begin{pmatrix}
        \zeros & \multiplierM^\transpose \\
        \multiplierM& \zeros
    \end{pmatrix}\begin{pmatrix}
        \smoothParam\allones_{\tau} & \smoothParam\closedloopG_1-\eye \\
        -\strgcvxParam\allones_{\tau} & \eye-\strgcvxParam\closedloopG_1
    \end{pmatrix}  \prec \zeros.
    \end{split}
\end{equation*} 
    The congruence transform $\diag(1, (\eye-\strgcvxParam\closedloopG_1)^{-1})$ is applied to the reformulated matrix inequality and yields
\begin{equation*}
    \begin{split}
        &(\bullet)^\transpose \begin{pmatrix}
        \zeros & \multiplierM^\transpose \\
        \multiplierM& \zeros
    \end{pmatrix}\begin{pmatrix}
        \smoothParam\allones_{\tau} & (\smoothParam\closedloopG_1-\eye)(\eye-\strgcvxParam\closedloopG_1)^{-1} \\
        -\strgcvxParam\allones_{\tau} & \eye
    \end{pmatrix}+ \\ &\quad (\bullet)^\transpose 
    \begin{pmatrix}
        -p\unitvec_1\unitvec_1^\transpose & \zeros \\
        \zeros & 1
    \end{pmatrix} \begin{pmatrix}
        \allones_{\tau} & \closedloopG_1(\eye-\strgcvxParam\closedloopG_1)^{-1} \\
        1 & \closedloopG_2(\eye-\strgcvxParam\closedloopG_1)^{-1}
    \end{pmatrix}  \prec \zeros.
    \end{split}
\end{equation*}
Furthermore, define the mapping $\closedloopG_1 \mapsto \youlaQ_1$
    \begin{equation}
    (\smoothParam\closedloopG_1-\eye)(\eye-\strgcvxParam\closedloopG_1)^{-1} = \youlaQ_1 - \eye, \label{eq:forward_transformation}
\end{equation}
with the inverse mapping given by
\(
    \closedloopG_1 = ((\smoothParam-\strgcvxParam)\eye + \strgcvxParam \youlaQ_1)^{-1}\youlaQ_1.    
\)
It is noted that this constitutes a bijective mapping from $\closedloopG_1 \in \linsysSet^1(1,\tau-1)$ to $\youlaQ_1\in \linsysSet^1(1,\tau-1)$, which is used to reformulate the above matrix inequality as
\begin{equation}
    \begin{split}
 &(\bullet)^\transpose \begin{pmatrix}
        \zeros & \eye \\
        \eye & \zeros
    \end{pmatrix}\begin{pmatrix}
    \smoothParam\multiplierM\allones_{\tau} &  \multiplierM\youlaQ_1- \multiplierM\\
        -\strgcvxParam\allones_{\tau} & \eye
    \end{pmatrix}+ \\
    &\quad(\bullet)^\transpose 
    \begin{pmatrix}
        -p\unitvec_1\unitvec_1^\transpose & \zeros \\
        \zeros & 1
    \end{pmatrix} \begin{pmatrix}
        \allones_{\tau} & \closedloopG_1(\eye-\strgcvxParam\closedloopG_1)^{-1} \\
        1 & \closedloopG_2(\eye-\strgcvxParam\closedloopG_1)^{-1}
    \end{pmatrix}  \prec \zeros.
    \end{split}
    \label{eq:matrix_ineq1}
\end{equation}
In addition, a second mapping $\closedloopG_2 \mapsto \youlaQ_2$ is defined through
\begin{equation}
    \youlaQ_2 = \closedloopG_2(\eye-\strgcvxParam\closedloopG_1)^{-1}, \text{ with the inverse } \closedloopG_2 = \youlaQ_2(\eye - \strgcvxParam\closedloopG_1).
    \label{eq:transformation2}
\end{equation}
Substituting equations~\eqref{eq:forward_transformation} and~\eqref{eq:transformation2} in inequality~\eqref{eq:matrix_ineq1} and multiplying out the first product while utilizing the fact that $\unitvec_1^\transpose\youlaQ_1 = 0$ (since $\youlaQ_1 \in \linsysSet^1(1,\tau-1)$) results in
\begin{equation}
    \begin{split}  
    (\bullet)^\transpose \begin{pmatrix}
        \zeros & \eye \\
        \eye & \zeros
    \end{pmatrix}&\begin{pmatrix}
    \smoothParam\multiplierM\allones_{\tau} &  \multiplierM\youlaQ_1- \multiplierM\\
        -\strgcvxParam\allones_{\tau} & \eye
    \end{pmatrix} -p\unitvec_1\unitvec_1^\transpose +\\
    & \quad \begin{pmatrix}
        1 \\ \youlaQ_2^\transpose
    \end{pmatrix}\begin{pmatrix}
        1 & \youlaQ_2
    \end{pmatrix}\prec \zeros.
    \end{split}
    \label{eq:temp_matrix_inequality}
\end{equation}
Two sources of nonlinearity remain in~\eqref{eq:temp_matrix_inequality}. The first results from the bilinear term $\multiplierM\youlaQ_1$. This nonlinearity is eliminated by defining the new variable $\youlaQ_\multiplierM:= \multiplierM\youlaQ_1$, which is strictly causal, because $\multiplierM\in \linsysSet^0(1,\tau-1)$ and $\youlaQ_1\in \linsysSet^1(1,\tau-1)$. 

The second nonlinearity is the rank-one term in~\eqref{eq:temp_matrix_inequality}, which is quadratic in $\youlaQ_2$. Applying the Schur complement lemma to this term shows that~\eqref{eq:temp_matrix_inequality} is equivalent to the LMI~\eqref{eq:cvx_sdp_constraint}. 

Lastly, to conclude the proof, we note that feasibility of the LMI~\eqref{eq:cvx_sdp_constraint} implies
$\He(\youlaQ_\multiplierM-\multiplierM)\prec0$. Since
$\youlaQ_\multiplierM\in \linsysSet^1(1,\tau-1)$ and
$\multiplierM\in\linsysSet^0(1,\tau-1)$, then $-2\multiplierM_{ii}<0$ for all $i$, which guarantees that $\multiplierM$ is nonsingular.

\end{proof}
It is noted that the synthesis problem is posed for a prescribed $\tau$. Its solution is a finite sequence of algorithm parameters designed to minimize the terminal distance after exactly $\tau$ iterations. Nevertheless, the terminal certificate can be used to construct an asymptotically convergent periodically restarted method.
\section{Simulation Results}
\label{sec:simulation}

\begin{figure*}[t]
    \centering
    \includegraphics[width = \textwidth]{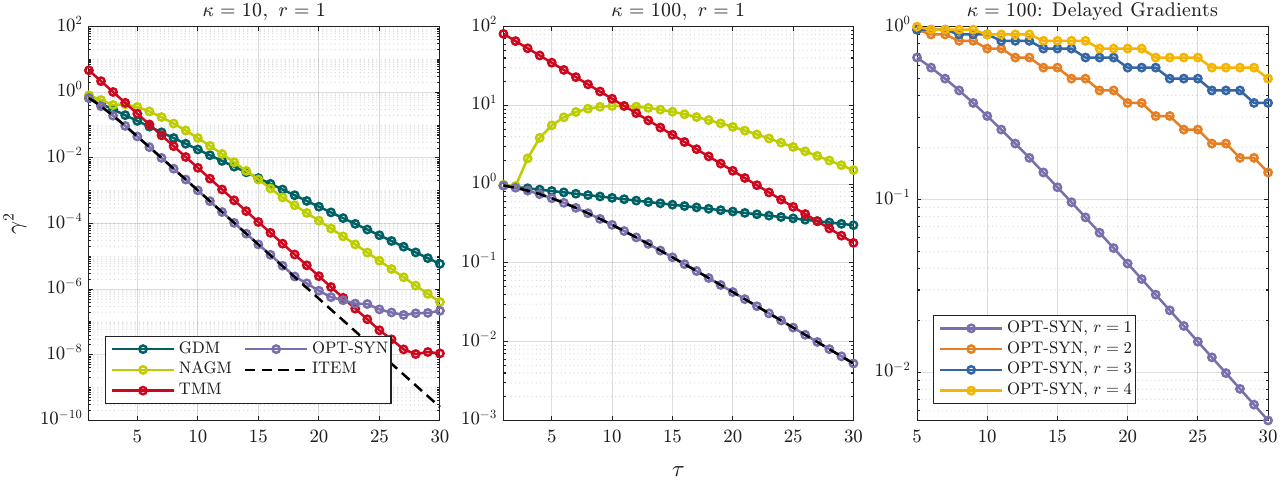}
    \caption{Performance comparison between the synthesized algorithms (denoted by OPT-SYN) and established first-order methods for different $\kappa$ and $r=1$. The right plot considers a fixed $\kappa = 100$ and relative degrees $r=1,2,3,4$ corresponding to delays of 0,1,2, and 3 sampling instants, respectively.}
    \label{fig:synthesis_problem_algorithms_comparison}
\end{figure*}

\noindent The synthesis Problem~\eqref{prob:UB_min_sup_io_problem3} is solved for
$\tau \in [1, 30]$ and $\kappa \in \{10, 100, 1000\}$ using
\textsc{MATLAB}/\textsc{YALMIP}~\cite{Lofberg2004} and \textsc{MOSEK}~\cite{mosek},
with $r = 1$. This yields a distinct time-varying
algorithm for every horizon; each point on the OPT-SYN curve, therefore, corresponds to
a different synthesized method. The synthesized algorithms (OPT-SYN) are benchmarked against the gradient descent method (GDM),
NAGM, the triple momentum method (TMM)~\cite{VanScoy2018}, and ITEM, where all benchmark methods are parameterized
for optimal asymptotic convergence on the function class $\smoothstrgcvxS_{\strgcvxParam, \smoothParam}$~\cite{Lessard_2016}. The upper bounds for the benchmark methods are
obtained by solving the corresponding analysis Problem~\eqref{prob:UB_min_sup_io_problem2}. The resulting upper bounds
$\gamma^2(\tau, \kappa)$ are shown in
Fig.~\ref{fig:synthesis_problem_algorithms_comparison}.

A key distinction between the methods is their transient behaviour. GDM guarantees
$\gamma^2 < 1$ for all $\tau \geq 1$,  albeit at a slower asymptotic convergence rate.
Momentum methods such as NAGM and TMM achieve superior asymptotic rates, but at the
cost of transient amplification: both exhibit $\gamma^2 > 1$ for small $\tau$,
indicating that the worst-case iterate distance to the optimum increases before
eventually decreasing. This transient cost grows more pronounced with increasing
$\kappa$ as the amplification region extends significantly, which is especially visible in the middle plot ($\kappa = 100$) of Figure~\ref{fig:synthesis_problem_algorithms_comparison}. Even TMM, which achieves the optimal asymptotic convergence rate, incurs a
non-negligible transient cost. In contrast to the quadratic case studied
in~\cite{mohammadi2023transient}, balanced initialization does
not eliminate worst-case transient amplification for NAGM over the class
$\mathcal{S}_{\mu,L}$.

The synthesized algorithms maintain $\gamma^2 < 1$
for all tested $\tau$ and $\kappa$. For any fixed finite-time horizon $\tau$, the synthesized algorithms guarantee a reduction of the error \emph{after $\tau$ steps}, even in the worst case. Moreover, larger horizons yield stronger error reduction, i.e., the achieved decrease over a horizon $\tau_2$ is greater than that obtained with any shorter horizon $\tau_1 < \tau_2$. Furthermore, as shown in Figure~\ref{fig:synthesis_problem_algorithms_comparison}, the upper bounds of OPT-SYN coincide with those of ITEM and the synthesized parameters recover the ITEM parameters, numerically confirming that the proposed method attains the information-theoretic complexity lower bound of~\cite{drori2021oraclecomplexitysmoothstrongly}. Therefore, no algorithm within the considered class can achieve tighter worst-case guarantees.

For large $\tau$, the upper bound curves of the synthesized algorithm become parallel
to those of TMM across all tested condition numbers, as seen in the middle plot of
Figure~\ref{fig:synthesis_problem_algorithms_comparison}. This reflects the convergence
of the ITEM parameters to those of TMM as $\tau \to \infty$, while retaining the
advantage of eliminating the transient amplification incurred by TMM at short horizons.

Independent validation via PESTO confirms that the IQC bounds match the exact PEP bounds for GDM, NAGM, and TMM, implying zero duality gap~\cite{taylor2016smoothstronglyconvexinterpolation}. For $\kappa = 10$ (left plot in Figure~\ref{fig:synthesis_problem_algorithms_comparison}) and large $\tau$,
minor numerical irregularities are observed in the upper bounds of the synthesized
algorithm, attributed to the computed values approaching the numerical precision of
the solver as $\gamma^2$ tends toward zero.


For delayed gradients, i.e., $r>1$, the benchmark GDM no
longer satisfies $\gamma^2<1$ over the tested horizons,
whereas the synthesized methods retain a certified terminal
contraction; see Figure~\ref{fig:synthesis_problem_algorithms_comparison}.
The computed bounds exhibit a staircase profile: they remain
constant until an additional delayed gradient sample becomes
available and decrease at the corresponding update times. This reflects the information pattern imposed by the relative
degree.

\section{Conclusion}
\label{sec:conclusion}
\noindent This paper developed an input--output framework for analyzing and synthesizing first-order methods for smooth strongly convex problems over a finite horizon, focusing on the terminal distance (although other objectives are possible). The analysis reduces to solving a convex semidefinite program that certifies worst-case transient performance. The same IQC-based formulation enables convex synthesis of a transient-optimal algorithm. Using time-varying Zames--Falb multipliers, the synthesized method guarantees descent after any horizon length  and recovers the ITEM parameters, thereby achieving the information-theoretic complexity lower bound. This shows that finite-horizon optimal performance, previously obtained via Performance Estimation Problems, can equally be realized through Zames--Falb IQC synthesis. Moreover, the framework seamlessly extends to problems with feedback delay, where the optimal algorithm respects the inherent latency and achieves the best possible finite-horizon performance.

Future work includes extending the framework to state-space realizations to address two practical limitations of the current input-output approach. First, the SDPs involve matrix variables of dimension $\tau \times \tau$, making computation prohibitive for large horizon lengths. Second, the synthesized algorithm requires access to the complete gradient history through the dense matrix $G_1$, incurring a memory footprint that grows with $\optdim\tau$ and hence is impractical in high dimensions.

\addtolength{\textheight}{-0cm}   




\appendix
\subsection{Proof of Lemma~\ref{lemma:h2_equivalence_error_coordinates}}
\label{appendix:proof_lemma}
\noindent Let $f \in \smoothstrgcvxS_{\strgcvxParam, \smoothParam}$ with minimizer $z_\star$, and define $h(\bullet) := f(\bullet + z_\star)$. Due to translation invariance, $h \in \smoothstrgcvxS_{\strgcvxParam,\smoothParam}$ with $\nabla h(0) = 0$, and hence $h \in \smoothstrgcvxS^0_{\strgcvxParam,\smoothParam}$.

Define $\delta \unceIn := \unceIn - \zSignal_\star$ and $\delta\perfOut := \perfOut - \boldclosedloopG_{\perfOut\perfIn}z_\star$. Substituting into the closed-loop equations of~\eqref{prob:min_sup_general_analysis} and removing constant terms yields the equivalent shifted system with $\delta \xi: = \xi-z_\star$.
Since $\boldclosedloopG_{\perfOut\perfIn} = \allones_\tau\otimes\eye_\optdim$, we have
$\boldclosedloopG_{\perfOut\perfIn}z_\star = \zSignal_\star$ and hence
$\delta\perfOut = \perfOut - \zSignal_\star$, so the objective
of~\eqref{prob:min_sup_general_analysis} coincides with
$\mathbb{E}\langle \delta\perfOut, \boldsymbol{R}_{\perfOut\perfOut}\delta\perfOut\rangle$.
Moreover, $\mathbb{E}[\xi] = z_\star, \operatorname{Cov}(\xi) = \boldsymbol{\Sigma_0}$ if and only if
$\mathbb{E}[\delta\xi] = 0, \operatorname{Cov}(\delta \xi) = \boldsymbol{\Sigma_0}$, and $f \mapsto h$ maps
$\{f\in\smoothstrgcvxS_{\strgcvxParam,\smoothParam} : \argmin f = z_\star\}$ bijectively
onto $\smoothstrgcvxS^0_{\strgcvxParam,\smoothParam}$. Therefore, the feasible sets correspond one-to-one with equal objective values.

\subsection{Proof of Theorem~\ref{cor:io_sdp_relaxation}}
\label{appendix:proof_cor}
\noindent The proof proceeds in three steps. 
First, the constraint $\wSignal =  \lifth(\zSignal)$ is replaced by the 
quadratic constraint
\begin{equation*}
0 \leq \mathbb{E}
\begin{pmatrix} \zSignal \\ \wSignal \end{pmatrix}^\transpose
\boldsymbol{\Pi}_{\multiplierM}
\begin{pmatrix} \zSignal \\ \wSignal \end{pmatrix}.
\end{equation*}
It is noted that this condition holds only in expectation and is, therefore, less strict than the deterministic IQCs in~\eqref{eq:quadratic_form_ozf_multiplier}.  
Moreover, replacing the exact functional relation by a quadratic constraint on the input-output signals represents a relaxation, since any
pair $(\zSignal, \wSignal)$ generated by an admissible gradient satisfies
\eqref{eq:quadratic_form_ozf_multiplier}, but the converse need not hold.
Hence the worst-case cost over the relaxed feasible set cannot be smaller
than over the original set, and the relaxed problem provides an
\emph{upper bound}.

In the second step, the input-output realization is substituted into the
objective and the quadratic constraint. Recall that, after the change to error coordinates, the notation
$\xi$ is used for the centered initial error $\delta\xi$. The problem then depends on
$(\zSignal, \wSignal)$ only through the second-order moment matrix
\[
\boldsymbol{\Sigma}:=
\begin{pmatrix}
  \mathbb{E}[\xi\xi^\transpose  ] & \mathbb{E}[\xi\mathbf{w}^\transpose  ] \\
  \mathbb{E}[\mathbf{w}\xi^\transpose  ] & \mathbb{E}[\mathbf{w}\mathbf{w}^\transpose  ]
\end{pmatrix}
=
\begin{pmatrix}
  \boldsymbol{\Sigma_0} & \boldsymbol{\Sigma}_{\zSignal\wSignal} \\
  \boldsymbol{\Sigma}_{\zSignal\wSignal}^\transpose   & 
  \boldsymbol{\Sigma}_{\wSignal\wSignal}
\end{pmatrix}\succeq 0,
\]
yielding the problem
\begin{mini*}
    {\boldclosedloopG_{\zSignal\unceOut},\,
     \boldclosedloopG_{\ySignal\unceOut}}
    {\sup_{\boldsymbol{\Sigma}_{\zSignal\wSignal},\,
           \boldsymbol{\Sigma}_{\wSignal\wSignal}}
     \left\langle \boldsymbol{R}_{\perfOut\perfOut},
     \begin{pmatrix}
       \boldclosedloopG_{\ySignal\perfIn}^\transpose \\
       \boldclosedloopG_{\ySignal\unceOut}^\transpose
     \end{pmatrix}^\transpose
     \boldsymbol{\Sigma}
     \begin{pmatrix}
       \boldclosedloopG_{\ySignal\perfIn}^\transpose \\
       \boldclosedloopG_{\ySignal\unceOut}^\transpose
     \end{pmatrix}
     \right\rangle}
    {}{}
    \addConstraint{
      \left\langle \boldsymbol{\Pi}_{\multiplierM},
      \begin{pmatrix}
        \boldclosedloopG_{\zSignal\perfIn}^\transpose & 0 \\
        \boldclosedloopG_{\zSignal\unceOut}^\transpose & \eye
      \end{pmatrix}^\transpose
      \boldsymbol{\Sigma}
      \begin{pmatrix}
        \boldclosedloopG_{\zSignal\perfIn}^\transpose & 0 \\
        \boldclosedloopG_{\zSignal\unceOut}^\transpose & \eye
      \end{pmatrix}
      \right\rangle}{\geq 0,}
    \addConstraint{\boldsymbol{\Sigma}}{\succeq 0.}
\end{mini*}

In the third step, the inner maximization over
$\boldsymbol{\Sigma}_{\zSignal\wSignal}$ and
$\boldsymbol{\Sigma}_{\wSignal\wSignal}$ is replaced by the Lagrange dual to obtain an upper
bound. A multiplier $\boldsymbol{P}$ is introduced for the
equality constraint that the $(1,1)$-block of $\boldsymbol{\Sigma}$ equals
$\boldsymbol{\Sigma_0}$. Since the IQC
constraint~\eqref{eq:quadratic_form_ozf_multiplier} holds for all
$\multiplierM \in \dhmatricesSet_\tau$ and $\dhmatricesSet_\tau$ is a
cone, $\iqcmultiplier_\multiplierM$ serves directly as the conic dual
variable with no additional scalar multiplier required. The Lagrangian is
\begin{align*}
&\mathcal{L}
=
\left\langle \boldsymbol{P}\!,\left(\boldsymbol{\Sigma}_{0} -
  \begin{pmatrix} \eye & 0 \end{pmatrix}
  \boldsymbol{\Sigma}
  \begin{pmatrix} \eye & 0 \end{pmatrix}^\transpose
\right)\right\rangle+ \\
&\quad 
\left\langle
\begin{pNiceMatrix}
  \boldclosedloopG_{\ySignal\perfIn} & \boldclosedloopG_{\ySignal\unceOut}\\ \hline
  \boldclosedloopG_{\zSignal\perfIn} & \boldclosedloopG_{\zSignal\unceOut}\\
  0 & \eye
\end{pNiceMatrix}^\transpose
\begin{pmatrix}
  \boldsymbol{R}_{\perfOut\perfOut} & 0\\
  0 & \iqcmultiplier_\multiplierM
\end{pmatrix}
\begin{pNiceMatrix}
  \boldclosedloopG_{\ySignal\perfIn} & \boldclosedloopG_{\ySignal\unceOut}\\ \hline
  \boldclosedloopG_{\zSignal\perfIn} & \boldclosedloopG_{\zSignal\unceOut}\\
  0 & \eye
\end{pNiceMatrix},
\boldsymbol{\Sigma}
\right\rangle.
\end{align*}
Since $\mathcal{L}$ is linear in $\boldsymbol{\Sigma}$, its supremum over
$\boldsymbol{\Sigma} \succeq 0$ is finite if and only if the combined
coefficient of $\boldsymbol{\Sigma}$ is negative semidefinite; otherwise
$\boldsymbol{\Sigma}$ can be scaled to drive $\mathcal{L}$ to $+\infty$.
Hence, enforcing finiteness is equivalent to
constraint~\eqref{eq:input_output_analysis_problem}.

Once this constraint holds, the supremum is attained at
$\boldsymbol{\Sigma} = 0$ and the Lagrangian reduces to
$\langle \boldsymbol{P}, \boldsymbol{\Sigma_0} \rangle =
\trace(\boldsymbol{P})$, since $\boldsymbol{\Sigma_0} = I$
by assumption. By weak duality, the resulting minimization over
$\multiplierP$ and $\multiplierM \in \dhmatricesSet_\tau$
subject to~\eqref{eq:input_output_analysis_problem} is exactly
Problem~\eqref{prob:UB_min_sup_io_problem2}, whose optimal value is an
upper bound on the worst-case cost.



\printbibliography

@misc{drori2012performancefirstordermethodssmooth,
  author       = {Drori, Yoel and Teboulle, Marc},
  title        = {Performance of first-order methods for smooth convex
                  minimization: a novel approach},
  year         = {2012},
  eprint       = {1206.3209},
  eprinttype   = {arXiv},
  primaryclass = {math.OC}
}

@misc{taylor2016smoothstronglyconvexinterpolation,
  author       = {Taylor, Adrien B. and Hendrickx, Julien M. and Glineur, Fran\c{c}ois},
  title        = {Smooth strongly convex interpolation and exact worst-case
                  performance of first-order methods},
  year         = {2016},
  eprint       = {1502.05666},
  eprinttype   = {arXiv},
  primaryclass = {math.OC}
}

@inproceedings{PESTO,
  author    = {Taylor, Adrien B. and Hendrickx, Julien M. and Glineur, Fran\c{c}ois},
  title     = {Performance estimation toolbox ({PESTO}): Automated worst-case
               analysis of first-order optimization methods},
  booktitle = {Proc. IEEE Conf. Decision Control (CDC)},
  year      = {2017},
  pages     = {1278--1283}
}

@misc{drori2021oraclecomplexitysmoothstrongly,
  author       = {Drori, Yoel and Taylor, Adrien},
  title        = {On the oracle complexity of smooth strongly convex minimization},
  year         = {2021},
  eprint       = {2101.09740},
  eprinttype   = {arXiv},
  primaryclass = {math.OC}
}

@misc{taylor2022optimalgradientmethodsmooth,
  author       = {Taylor, Adrien and Drori, Yoel},
  title        = {An optimal gradient method for smooth strongly convex minimization},
  year         = {2022},
  eprint       = {2101.09741},
  eprinttype   = {arXiv},
  primaryclass = {math.OC}
}

@inproceedings{DURR2012291,
  author    = {D{\"u}rr, Hans-Bernd and Ebenbauer, Christian},
  title     = {On a class of smooth optimization algorithms with applications
               in control},
  booktitle = {Proc. IFAC Conf. Nonlinear Model Predictive Control},
  year      = {2012},
  pages     = {291--298}
}

@inproceedings{multidimensionalHB,
  author    = {Michalowsky, Simon and Ebenbauer, Christian},
  title     = {The multidimensional $n$-th order heavy ball method and its
               application to extremum seeking},
  booktitle = {Proc. IEEE Conf. Decision Control (CDC)},
  year      = {2014},
  pages     = {2660--2666}
}

@article{Lessard_2016,
  author  = {Lessard, Laurent and Recht, Benjamin and Packard, Andrew},
  title   = {Analysis and design of optimization algorithms via integral
             quadratic constraints},
  journal = {SIAM J. Optim.},
  volume  = {26},
  number  = {1},
  pages   = {57--95},
  year    = {2016}
}

@article{convex_synthesis,
  author  = {Scherer, Carsten W. and Ebenbauer, Christian},
  title   = {Convex synthesis of accelerated gradient algorithms},
  journal = {SIAM J. Control Optim.},
  volume  = {59},
  number  = {6},
  pages   = {4615--4645},
  year    = {2021}
}

@article{scherer2025tutorial,
  author  = {Scherer, Carsten W. and Ebenbauer, Christian},
  title   = {A tutorial on convex design of optimization algorithms by
             integral quadratic constraints},
  journal = {Annu. Rev. Control Robot. Auton. Syst.},
  volume  = {9},
  year    = {2025},
  note    = {In press}
}

@article{mohammadi2023transient,
  author  = {Mohammadi, Hesameddin and Samuelson, Samantha and
             Jovanovi{\'c}, Mihailo R.},
  title   = {Transient growth of accelerated optimization algorithms},
  journal = {IEEE Trans. Autom. Control},
  volume  = {68},
  number  = {3},
  pages   = {1823--1830},
  month   = mar,
  year    = {2023},
  doi     = {10.1109/TAC.2022.3162154}
}

@inproceedings{FETZER2017,
  author    = {Fetzer, Matthias and Scherer, Carsten W.},
  title     = {Absolute stability analysis of discrete time feedback
               interconnections},
  booktitle = {Proc. IFAC World Congress},
  year      = {2017},
  pages     = {8447--8453}
}

@misc{scherer2023robustexponentialstabilityinvariance,
  author       = {Scherer, Carsten W.},
  title        = {Robust exponential stability and invariance guarantees with
                  general dynamic {O'Shea--Zames--Falb} multipliers},
  year         = {2023},
  eprint       = {2306.00571},
  eprinttype   = {arXiv},
  primaryclass = {math.OC}
}

@article{VanScoy2018,
  author  = {Van Scoy, Bryan and Freeman, Randy A. and Lynch, Kevin M.},
  title   = {The fastest known globally convergent first-order method for
             minimizing strongly convex functions},
  journal = {IEEE Control Syst. Lett.},
  volume  = {2},
  number  = {1},
  pages   = {49--54},
  month   = jan,
  year    = {2018}
}

@inproceedings{Lofberg2004,
  author    = {L{\"o}fberg, Johan},
  title     = {{YALMIP}: A toolbox for modeling and optimization in {MATLAB}},
  booktitle = {Proc. CACSD Conf.},
  year      = {2004}
}

@manual{mosek,
  author  = {{MOSEK ApS}},
  title   = {The {MOSEK} Optimization Toolbox for {MATLAB}, Version 11.0},
  year    = {2025},
}

@mastersthesis{adel2025,
  author  = {Abdelsamed, Adel},
  title   = {Analysis and design of the transient behavior of convex
             optimization algorithms using control theory},
  school  = {RWTH Aachen University},
  address = {Aachen, Germany},
  type    = {{M.S.} thesis},
  year    = {2025}
}

\end{document}